\documentclass[12pt]{amsart}
\usepackage{preamble}

\title{Increasing the Size of Tame Shafarevich Groups}
\author{Andreea Iorga, Ravi Ramakrishna}
\address{Department of Mathematics, Cornell University, Ithaca, USA}

\email{ai324@cornell.edu rkr5@cornell.edu}

\thanks
{The authors thank Farshid Hajir and Christian Maire for helpful conversations,  and
Ariane M\'{e}zard for  detailed comments and corrections on  an early draft of the paper.
The second author was supported by a Simons Travel support gift. The authors used ChatGPT and Claude to assist in identifying typos, grammatical mistakes and notational inconsistencies.}

\begin{document}
\begin{abstract}
    Let $K$ be a number field with $S$ a finite set of primes. We study the cohomology of $\F_p[G_{K,S}]$-modules $A$, in particular the Shafarevich groups $\Sha^i_S(K,A)$ for $i=1,2$ and tame sets $S$, i.e., for sets $S$ that contain no primes above $p$. When $S$ contains all primes above $p$ (the ``wild'' setting), it is a consequence of global Poitou--Tate duality that 
    $$\Sha^1_S(K,A')^\vee \simeq \Sha^2_S(K,A) \stackrel{\simeq}{\hookrightarrow} \RusB_S(K,A) $$ is non-increasing  as $S$ increases. A similar result holds when $G_{K,S}$ is replaced by its maximal pro-$p$ quotient $G_{K,S}(p)$. 
    In \cite{MR4359923} it was shown that for $S$ tame and $A=\F_p$ with trivial action, the group $\Sha^2_S(K, \F_p)$ can {\it increase}  as $S$ increases to $S\cup X$, and even attain its maximal dimension, $\dim \RusB_S(K,\F_p)$, for carefully chosen $X$. In the first part of this paper, we use Liu's definition \cite{MR4896734} of $\RusB_S(K,A)$ for a general $\F_p[G_{K, S}]$-module $A$ to show, assuming $\Sha^1_{all}(K,A')=0$, that $\Sha^2_S(K,A) \hookrightarrow \RusB_S(K,A)$. This happens, for example, when the action of $G_{K,S}$ on $A$ is through a finite group of order prime to $p$. Under this extra assumption, we then strengthen the results of \cite{MR4359923} to show that for any odd prime $p$ and any $\F_p[G_{K, S}]$-module $A$ with $S$ tame, there exist infinitely many tame sets of primes $X$ of $K$ such that $\Sha^2_{S\cup X}(K,A) \stackrel{\simeq}{\hookrightarrow} \RusB_{S \cup X}(K,A) \stackrel{\simeq}{\twoheadleftarrow} \RusB_S(K,A) \hookleftarrow \Sha^2_S(K,A)$.
    
\end{abstract}

\maketitle

\section{Introduction}\label{intro}
Let $K$ be a number field with an algebraic closure $\overline{K}$ and set $G_K=\Gal(\overline{K}/K)$.
Let $S$ be a finite set of primes of $K$, and $p$ a fixed rational prime.
Denote by $K_S \subset \overline{K}$ the maximal extension of $K$ unramified outside $S$ and 
by $K_S(p) \subset K_S$ the maximal pro-$p$ extension of $K$ unramified outside $S$, with 
$G_{K, S}$ and $G_{K,S}(p)$  the corresponding Galois groups over $K$.  
Let $K_{\q}$ be the completion of $K$ at the prime ideal ${\q}$ and let $G_{K_{\q}}$ be an absolute Galois group of $K_{\q}$ with $G_{K_{\q}}(p)$ its maximal pro-$p$ quotient.
For a finite cardinality $G_{K,S}$-module $A$,
let 
\[
\Sha_{S}^i (K, A) = \ker \left(H^i(G_{K, S}, A) \to \prod_{{\q} \in S} H^i(G_{K_{\q}}, A)\right),
\]
as usual. As $A$ decomposes as a direct sum of $p$-primary modules and cohomology commutes with direct sums, it suffices to assume $A$ is a finite cardinality $\Z_p[G_{K,S}]$-module.

When $S$ contains all primes dividing $p$ (the ``wild'' setting), Poitou-Tate duality gives that 
$\Sha_{S}^2 (K, A) \simeq \Sha_S^1(K,A')^\vee$
where $A':=\Hom(A,\mu_{p^\infty})$ and $\mu_{p^\infty}$ is the group of all $p$-power roots of unity.
There is also a global Euler-Poincar\'{e} characteristic formula:
\[
\frac{\#H^0(G_{K,S},A) \cdot \# H^2(G_{K,S},A)}{\#H^1(G_{K,S},A)} =
\prod_{\mathfrak q | \infty} \frac{\# H^0(G_{K_{\mathfrak q}},A)}{|\#A|_{\mathfrak q}}.
\]
As  $H^0$s  are  ``easy'' to calculate,
this may be viewed as an $H^1$--$H^2$ duality: knowledge of the order of one of these cohomology groups gives the order of the other. 
A key ingredient in the proof of the wild Poitou-Tate duality is 
the exact Kummer sequence below, where
${\mathcal O}^\times_{K_S,S}$ is the group of $S$-units of  $K_S$:
\[
1 \rightarrow \mu_p \rightarrow {\mathcal O}^\times_{K_S,S} \stackrel{x\mapsto x^p}{\longrightarrow} {\mathcal O}^\times_{K_S,S} \to 1.
\]
If $S$ is tame, that is, if $S$ contains no primes above $p$, the $p$th power map need not be surjective, as extracting a $p$th root usually introduces ramification at $p$. This is a serious obstruction to an $H^1$-$H^2$ duality.

Our goal is to establish a partial duality result  in the tame case.  
Recall the following definitions:
\[
V_S(K,\F_p):= \{ x \in K^\times \, | \,  v_{\q}(x) \equiv 0 \bmod{p}, \,\,\forall {\q}; \mbox{ }x\in K^{\times p}_{\q}, \,\,\forall {\q} \,\in S\}
\]
and 
{the dual of $V_S(K, \F_p)/K^{\times p}$:}
\[\RusB_S(K, \F_p):=(V_S(K,\F_p)/K^{\times p})^\vee.
\] 
It is well known that for arbitrary $S$ we have the injection $\Sha^2_S(K, \F_p) \hookrightarrow \RusB_S(K, \F_p)$, and that this injection is an isomorphism if $S$ contains all primes above $p$. Furthermore, 
both $V_S(K,\F_p)$ and $\RusB_S(K, \F_p)$ are non-increasing as $S$ increases, and
$\dim \RusB_S(K, \F_p) < \infty$ for all sets $S$:
letting $\Cl_K[p]$ denote the $p$-torsion subgroup of the class group of $K$ and $U_K$ the units of $K$, we have the  exact sequence:

\begin{equation}
\label{eq:Vemptyset}
0 \to \frac{U_K}{U^p_K} \to 
\frac{V_\emptyset(K,\F_p)}{K^{\times p}} \to \Cl_K[p] \to 0.
\end{equation}
Thus understanding $\Sha^2_\emptyset(K, \F_p) \hookrightarrow \RusB_\emptyset(K,\F_p)$ involves, at a minimum, disentangling the units and class group of $K$. For references, see \S 11.2 and \S 11.3 of \cite{MR1930372} and Theorem 8.6.7 and Proposition 10.7.2 of \cite{MR2392026}.

Building on the work of Labute \cite{MR2254811}, Schmidt \cite{MR2629694} has proved that for an arbitrary number field $K$ and an odd prime $p$, there exists a finite set $Y$ of tame primes such that  $G_{K,Y}(p)$ has $p$-cohomological dimension $2$, and its
Galois cohomology
agrees with the \'{e}tale cohomology of the pro-$p$ cover of $\Spec({\mathcal O}_{K} \setminus Y)$. Schmidt's proof relies on, among other things,  choosing $Y$ so that $\RusB_Y(K, \F_p)$ and $\Sha^2_Y(K,\F_p)$ are trivial.  
 
For a general tame set $S$, it is very difficult to have any control over the group $\Sha_S^2(K, \F_p)$, and therefore to decide if $\Sha_S^2(K, \F_p) \hookrightarrow \RusB_S(K, \F_p)$ is a proper inclusion or an isomorphism. Setting
\[
\Sha_{S, p}^2(K, \F_p) := \ker\left(H^2(G_{K, S}(p), \F_p) \to \prod_{{\q} \in S} H^2(G_{K_{\q}}(p), \F_p)\right)
\]
and using Proposition 20 in \cite[II.\S 5.6]{MR1324577}, we have (replacing the local pro-$p$ Galois groups by the full groups):
\[
\Sha_{S, p}^2(K, \F_p) = \ker\left(H^2(G_{K, S}(p), \F_p) \to \prod_{{\q} \in S} H^2(G_{K_{\q}}, \F_p)\right).
\]
One checks  $\Sha_{S, p}^2(K, \F_p) \hookrightarrow \Sha_{S}^2(K, \F_p)$ (see Lemma~\ref{lemma:ShaInj})
and then one can construct examples in which the cokernel of
$\Sha_{S,p}^2(K, \F_p) \hookrightarrow \RusB_S(K, \F_p)$ is large. In \cite{MR4359923}, 
 it is shown that if $S$ is a (possibly empty) tame set of primes of $K$ and $p$ is odd, then there exist infinitely many finite tame sets $X$ such that
\begin{equation}
\label{eq:H1formula}
\Sha_{S \cup X, p}^2(K, \F_p) \stackrel{\simeq}{\hookrightarrow} \Sha^2_{S\cup X}(K, \F_p) \stackrel{\simeq}{\hookrightarrow} \RusB_{S \cup X}(K, \F_p) \stackrel{\simeq}{\twoheadleftarrow} \RusB_S(K, \F_p). \end{equation}
A slightly weaker result is obtained for $p=2$.
In contrast to Schmidt's result,  $X$ is chosen so 
$\Sha^2_{S\cup X}(K, \F_p)$ is {\it maximized} as opposed to trivialized. As in Schmidt's work, the primes of $X$ must be chosen to split appropriately in a certain governing field. 
For a definition of governing fields and a more detailed discussion, we refer the reader to \cite[Chapter V]{MR1941965}.
Recall that for a prime $p$ and number field $F$ we define $\delta_p(F):=\left\{ \begin{array}{cc} 1 & \mu_p \subset F\\ 0 & \mu_p \not \subset F \end{array} \right.$, and for a prime ${\q}$  we set $\chi({\q})$ to be the residue characteristic of ${\q}$. Let $r_1$ and $r_2$ be the number of real and pairs of complex embeddings of $F$.
From Theorem 11.8 of \cite{MR1930372} we have
\begin{equation}
\label{eq:ActualH1formula}
\dim H^1(G_{K,Z},\F_p) =
\sum_{\substack{{\q} \in Z \\ \chi({\q})=p}} [K_{\q}:{\q}_p] -\delta(K)-r_1-r_2+1+
\sum_{{\q} \in Z}\delta_p(K_{\q}) + \dim \RusB_Z(K,\F_p),
\end{equation}
so when~(\ref{eq:H1formula}) holds,
we have a partial tame $H^1$-$H^2$ duality.

In \cite{MR4896734}, Liu generalizes the definition
$\RusB_S(K, \F_p)$ to
finite cardinality $\Z_p[G_{K,S}]$-modules $A$. The dual of the 
map on the right below is, using local duality,  just the restriction map.
\[
\RusB_S(K, A) := \coker\left(\prod_{{\q} \in S} H^1(G_{K_{\q}}, A) \times \prod_{{\q} \notin S} H^1_{nr}(G_{K_{\q}}, A) \to H^1(G_K, A
')^\vee \right).
\]
Note that from the definition it is clear that $\RusB_S(K, A)$ is non-increasing as $S$ increases.
When $S$ is wild, it is a straightforward computation in local duality {(\cite[Lemma 8.7]{MR4896734})} to see that $\RusB_S(K, A)$ is dual to $\Sha_S^1(K,A')$. We use the subscript ``all'' to indicate $S$ consists of all places of $K$.

Our first result  generalizes the classical injection
$\Sha^2_S(K, \F_p) \hookrightarrow \RusB_S(K, \F_p)$ in certain situations. Note that our notation differs from that of \cite{MR4896734}. In the proof we give a dictionary to aid the reader in understanding the relevant part of that paper. 

\begin{theorem}\label{thm:Shaisinrusb} 
    Let $A$ be any finite cardinality $\Z_p[G_{K,S}]$-module and let $L:=K(A)$ be the splitting field of $A$.
    \begin{enumerate}[(i)]
        \item If $\Sha^1_{all}(K, A') = 0$, then there is a natural injection $\Sha^2_S(K, A) \hookrightarrow \RusB_S(K, A)$.
        \item Whenever we have an injection $\Sha^2_S(K, A) \hookrightarrow \RusB_S(K, A)$, it is  $\Gal(L/K)$-equivariant.
    \end{enumerate}
\end{theorem} 
In \cite{MR4896734}, an unconditional version of {\it (i)} is established $K$-theoretically, that is both sides are acted on by $\Gal(L/K)$ with $(\#\Gal(L/K),p)=1$. Liu shows that the multiplicity of any irreducible constituent on the left side is  less than or equal to its multiplicity on the right side. 
When $A=\F_p$ with trivial Galois action,
the result becomes 
$\dim \Sha^2_S(K, \F_p) \leq \dim \RusB_S(K, \F_p)$. 
We establish an actual injection under the above conditions. We also discuss when the assumption of {\it (i)} above is satisfied.  

For a finite cardinality $\Z_p[G_{K}]$-module $A$ with splitting field $L$, let $\Gamma=\Gal(L/K)$. If $S$ is a set of primes of $K$, let $\tilde{S}$ be the set of primes of $L$ above those in $S$. Note that if $A$ is a finite cardinality $\Z_p[G_{L, \tilde{S}}(p) \rtimes \Gamma]$-module with the property that $(\#\Gamma, p) = 1$, the fields $L$ and $K_S(p)$ are linearly disjoint over $K$. It follows that there is no action of $G_{K, S}(p)$ on $A$, so the object $\Sha^2_{S, p}(K, \F_p)$ (or $\Sha^2_{S, p}(K, A)$) does not exist. In this situation, we make the following definition: 
\begin{definition} Consider the semidirect product
    $G_{L, \tilde{S}}(p) \rtimes \Gamma$, given by the conjugation action of the prime-to-$p$ group $\Gamma$ on the pro-$p$ group $G_{L, \tilde{S}}(p)$. For a finite 
    {cardinality} 
    $\Z_p[G_{L, \tilde{S}}(p) \rtimes \Gamma]$-module $A$, define the following group: 
    $$\Sha^2(L_{\tilde{S}}(p)/K, A) := \ker\left(H^2(G_{L, \tilde{S}}(p) \rtimes \Gamma, A) \to \prod_{\q \in S}H^2(G_{K_\q}, A)\right).$$
    Note that the map above is obtained by restricting $H^2(G_{L, \tilde{S}}(p) \rtimes \Gamma, A)$ to the decomposition subgroups at $\q$ and then inflating to $H^2(G_{K_\q}, A)$. 
\end{definition}
When $\tilde{S}$ is wild,  the cohomology theory of such groups was developed in \cite{MR2470687} as part of the proof of the Sato-Tate conjecture, and later in \cite{MR4896734} to obtain results about the number of generators of profinite Galois groups.

The main results of this paper are the following generalizations of \cite[Theorem A]{MR4359923}: 
\begin{prop}\label{thm: main}
Let $p$ be an odd prime and let $A$ be a finite 
{cardinality} $\F_p[G_K]$-module. Recall $L=K(A)$ is the splitting field of $A$.
Let $S$ be a finite set of tame primes of $K$. Then there exist infinitely many finite  tame sets $X$ of places of $K$ such that 
    \begin{align*}
        (\Sha^2_{\tilde{S} \cup \tilde{X}, p}(L, A))^\Gamma \stackrel{\simeq}{\hookrightarrow}
        (\Sha^2_{\tilde{S} \cup \tilde{X}}(L,A))^\Gamma\stackrel{\simeq}{\hookrightarrow} 
        \RusB_{\tilde{S} \cup \tilde{X}}(L, A)^\Gamma \stackrel{\simeq}{\twoheadleftarrow} 
        \RusB_{\tilde{S}}(L, A)^\Gamma.  
    \end{align*}
\end{prop}
\begin{theorem}\label{thm:main2}
    With the notations as above, suppose $S$ contains all ramified primes of $L/K$ and $(\#\Gamma,p)=1$. Then there exist infinitely many finite  tame sets $X$ of places of $K$ such that 
    {\begin{align*}
        \Sha^2(L_{\tilde{S}\cup \tilde{X}}(p)/K, A) \stackrel{\simeq}{\hookrightarrow}
        \Sha^2_{S \cup X}(K, A) \stackrel{\simeq}{\hookrightarrow} \RusB_{S \cup X}(K, A) \stackrel{\simeq}{\twoheadleftarrow} \RusB_{S}(K, A) \hookleftarrow \Sha^2(L_{\tilde{S}}(p)/K, A). 
    \end{align*}}
\end{theorem}
While Schmidt's work \cite{MR2629694} gives a very strong tame duality theorem by choosing $S$ such that $\RusB_S(K,\F_p)=0$, our results represent progress towards establishing a tame duality when $\RusB_S(K,A)\neq 0$. The obstruction to formulating an $H^1$-$H^2$ duality in this case lies in the  difficulty of determining the image of the localization map. 
Theorem~\ref{thm:main2} most naturally generalizes the main result of \cite{MR4359923}, 
but if the set of ramified primes of $L/K$ contains primes that are {\it not} in $S$, Proposition~\ref{thm: main} is the natural result.
In Section~\ref{MainResults}, we also prove that any isomorphism $\Sha^2(L_{\tilde{S}}(p)/K, A) \stackrel{\simeq}{\hookrightarrow} \Sha^2_S(K, A) \stackrel{\simeq}{\hookrightarrow} \RusB_S(K, A)$ is preserved if we increase the set $S$ (see Corollary~\ref{cor:isoarepreserved}).
\vskip1em
The method of proof of Proposition~\ref{thm: main} is to pass to the field $L$ where $A$ is a trivial Galois module. Using the main result of \cite{MR4359923} we obtain
\[
\begin{tikzcd}
    \Sha^2_{\tilde{S}, p}(L, \F_p^{\dim A}) \arrow[r, hook] & \Sha^2_{\tilde{S}}(L, \F_p^{\dim A}) \arrow[r, hook] & \RusB_{\tilde{S}}(L, \F_p^{\dim A}) \arrow[d, two heads, "\simeq"] \\ \Sha^2_{\tilde{S} \cup X', p}(L, \F_p^{\dim A}) \arrow[r, hook, "\simeq"] & \Sha^2_{\tilde{S} \cup X'}(L, \F_p^{\dim A}) \arrow[r, hook, "\simeq"] & \RusB_{\tilde{S} \cup X'}(L, \F_p^{\dim A})
\end{tikzcd}
\]
for infinitely many finite sets $X'$ of primes of $L$.
{The key ingredient of this proof is Lemma \ref{RusBpreserved}, which gives that the above isomorphisms persist when we replace
$X'$ by its $\Gal(L/K)$-orbit $\tilde{X}$.}
We then take $\Gamma$-invariants to obtain Proposition~\ref{thm: main}.
The descent to  $K$  in Theorem~\ref{thm:main2} uses that $(\#\Gamma,p)=1.$
The reason for working only with Galois modules that are vector spaces over $\F_p$ is because the results of  \cite{MR4359923} only hold for $A = \F_p$. To obtain our result for finite cardinality $\Z_p[G_K]$-modules would require extending those results to the trivial Galois modules $\Z/p^k\Z$.

In \S \ref{ShaRusB} we prove Theorem~\ref{thm:Shaisinrusb} and discuss when $\Sha^1_{all}(K,A')$ is trivial. 
In \S \ref{SomeLemmas} we prove some  lemmas needed in \S \ref{MainResults}.
In \S \ref{MainResults} we prove Proposition~\ref{thm: main} and Theorem~\ref{thm:main2}.
In \S \ref{p=2} we address the case $p=2$. The results of \cite{MR4359923} are slightly weaker in this case, leading to our slightly weaker results.
In \S \ref{examples} we give explicit examples.
\vskip1em
\noindent \textbf{Definitions and Notations}

\begin{itemize}[leftmargin=*]
  \item $p$:  A prime number.
  \item $K$: A number field.  
  \item $G_K$: The absolute Galois group of $K$.
  \item $\mathcal{O}_K$:  The ring of integers of $K$. 
  \item A prime ideal $\q$ of $\mathcal{O}_K$ is called tame if $\#\mathcal{O}_K/\q \equiv 1 \pmod{p}$. This is a necessary condition for a $p$-extension of $K$ to be ramified at $\q$.
  \item $K_S$, $K_S(p)$: If $S$ is a set of primes, $K_S$ (resp. $K_S(p)$) is the maximal extension of $K$ unramified outside $S$ (resp. the maximal pro-$p$ extension of $K$ unramified outside $S$).
  \item  $G_{K, S} := \Gal(K_S/K)$,  $G_{K, S}(p) := \Gal(K_S(p)/K)$. 
  \item $A$:  A finite {cardinality} $\Z_p[G_{K, S}]$-module with dual 
  $A' := \Hom(A, \mu_{p^\infty})$. 
  \item $L := K(A)$:  The splitting field of $A$, i.e., the subextension of $K_S/K$ such that $\Gal(K_S/L)$ is the maximal subgroup of $G_{K, S}$ which acts trivially on $A$. 
  \item $\Gamma:= \Gal(L/ K)$.
  \item $\tilde{T}$: For $T$ a set of primes of $K$, this is the set of primes of $L$ above $T$. 
  \item $K_\q$: For a prime ideal $\q$ of $\mathcal{O}_K$, this is the completion of $K$ at $\q$. 
  \item $G_{K_\q}$, $G_{K_\q}(p)$: The absolute Galois group of $K_\q$ and its  maximal pro-$p$ quotient.
  \item $\mathcal{I}_{\q}(K)$: The inertia subgroup of $G_{K_\q}$.
  \item $H^1_{nr}(G_{K_\q}, A) := \ker(H^1(G_{K_\q}, A) \to H^1(\mathcal{I}_\q(K), A))$. This is the unramified cohomology.
  \item  $\Sha_S^i(K, A):= \ker\left(H^i(G_{K, S}, A) \to \prod_{\q \in S} H^i(G_{K_\q}, A)\right)$.
  
  \item $\begin{aligned}[t]
      \Sha_{\tilde{S}, p}^2(L, A) :& = \ker\left(H^2(G_{L, \tilde{S}}(p), A) \to \prod_{{\tilde{\q}} \in \tilde{S}} H^2(G_{L_{\tilde{\q}}}(p), A)\right) \\ &\simeq \ker\left(H^2(G_{L, \tilde{S}}(p), A) \to \prod_{{\tilde{\q}} \in \tilde{S}} H^2(G_{L_{\tilde{\q}}}, A)\right).
  \end{aligned}$
  
  \noindent The isomorphism follows from Proposition 20 in \cite[II.\S 5.6]{MR1324577}.
  \item  $\Sha^2(L_{\tilde{S}}(p)/K, A) := \ker\left(H^2(G_{L,\tilde{S}}(p)\rtimes \Gamma, A) \to \prod_{\q \in S}H^2( G_{K_\q}, A)\right)$, for a finite cardinality $\Z_p[G_{K, S}]$-module $A$ for which $(\#\Gamma, p) = 1$.
  \item $\Sha^i_{all}(K, A) := \ker \left(H^i(G_{K}, A) \to \prod_{ \text{ all }\q} H^i(G_{K_\q}, A)\right).$
\item $\RusB_S(K, A) := \coker\left(\prod_{\q \in S} H^1(G_{K_\q}, A) \times \prod_{\q \notin S} H^1_{nr}(G_{K_\q}, A) \to H^1(G_K, A')^\vee \right)$, where the map is the dual of the restriction map, using local duality.
\item $\RusB_{all}(K, A) := \coker\left( \prod_{\text{ all }\q} H^1(G_{K_\q}, A) \to H^1(G_K, A')^\vee \right)$. {This group} is dual to $\Sha^1_{all}(K, A')$ {\cite[Lemma 8.7]{MR4896734}}. 
\item $V_S(K,\F_p):= \{ x \in K^\times \, |  \, v_\q(x) \equiv 0 \bmod{p}, \,\,\forall {\q}; \mbox{ }x\in K^{\times p}_\q, \,\,\forall \q \,\in S\}.$
{By \cite[Proposition 8.3]{MR4896734},} $(V_S(K, \F_p)/K^{\times p})^\vee \simeq \RusB_S(K, \F_p).$
\end{itemize}

\section{Proof of Theorem \ref{thm:Shaisinrusb}}\label{ShaRusB}

In this section, we prove Theorem~\ref{thm:Shaisinrusb}.  Let $p$ be any prime (including 2), and let $S$ be a finite set of primes of $K$.
\begin{proof}[Proof of Theorem~\ref{thm:Shaisinrusb}]
    Below are some remarks  translating between our notation and that of Liu
in \cite{MR4896734}. 
\begin{itemize} 
\item Our $K$ plays the role  of $Q$ in \cite{MR4896734}. The Galois action on $A$ always factors through $\Gal(L/K)$, whereas in
\cite{MR4896734} the action factors through, in that notation, $\Gal(k_S/Q)$.
\item $N:=\ker\left(\RusB_S(K, A) \twoheadrightarrow \RusB_{all}(K, A)\right)$.
\item $\beta : H^2(G_{K, S}, A) \to H^2(G_K, A)$ is the inflation map. 
\end{itemize}
 Towards the end of the proof of Proposition 8.5 of \cite{MR4896734}, there are two exact sequences 
\[
\begin{tikzcd}
    0 \arrow[r] & \ker(\beta) \arrow[r] & \Sha_S^2(K, A) \arrow[r, "\beta"] & \beta(\Sha^2_S(K, A)) \arrow[r] &0 \\
    0 \arrow[r] & N \arrow[r] & \RusB_S(K, A) \arrow[r] & \RusB_{all }(K, A) \arrow[r] &0
\end{tikzcd}
\]

An easy diagram chase in the proof of \cite[Proposition 8.5]{MR4896734} shows that $ \ker(\beta) \hookrightarrow N$ and $\beta(\Sha^2_S(K,A)) \hookrightarrow  \RusB_{all}(K,A)$. By \cite[Proposition 8.3]{MR4896734}, $\Sha^1_{all}(K, A') \cong \RusB_{all}(K, A)^\vee$, so our hypothesis implies that $\RusB_{all}(K,A)=0$, so 
$\beta(\Sha^2_S(K,A))=0$, and the diagram above becomes
\[
\begin{tikzcd}
    0 \arrow[r] & \ker(\beta) \arrow[r] \arrow[d, hook] & \Sha^2_S(K, A) \arrow[r] & 0 \arrow[r] & 0 \\
    0 \arrow[r] & N \arrow[r] &\RusB_S(K, A) \arrow[r] & 0 \arrow[r] & 0
\end{tikzcd}
\]
which gives $\Sha^2_S(K,A) \hookrightarrow \RusB_S(K,A)$, proving the first part.

For {\it (ii)}, Liu proves the injection 
$\ker(\beta) \hookrightarrow N$ is $\Gal(L/K)$-equivariant.
\end{proof}

We now investigate when the hypothesis of Theorem~\ref{thm:Shaisinrusb} is satisfied. We start by noting that there are several situations in which a finite cardinality $\Z_p[G_K]$-module $A$ gives $\Sha^1_{all}(K, A) = 0$ covered in \cite{MR2392026}. We include them here for completeness; for a more detailed discussion, see Theorem 9.1.9 and 9.1.15 in \cite{MR2392026}. 
Of course there are famous counterexamples due to Wang \cite{MR26992}. 
\begin{example} \label{Sha=0} (\cite[Theorems 9.1.9, 9.1.15]{MR2392026})
    Let $A$ be any finite cardinality $\Z_p[G_K]$-module. Then $\Sha^1_{all}(K, A) = 0$ in the following cases:
    \begin{enumerate}[(i)]
        \item $A$ is a trivial $G_K$-module.
        \item $p$ is odd and $A = \mu_{p^n}$.
        \item $p = 2$ and $A = \mu_{2^n}$, except in the special case where $n \geq 3$, 
        $K(\mu_{2^n})/K$ is not cyclic and all primes dividing $2$ decompose in $K(\mu_{2^n})/K$.
        Then $\Sha^1_{all}(K, A) \cong \Z/2\Z$. 
        \item $\#\Gal(L/K) = lcm\{\#\Gal(L_\q/K_\q), \q \text{ is a prime of } K\}$.
        \item $pA = 0$, $A$ is a simple $G_K$-module and $\Gal(K(A)/K)$ is solvable. 
        \item $pA = 0$, $A$ is a simple $G_K$-module and there exist finite extensions $K' \subset K''$ of $K$ such that $[K'':K']$ is prime to $p$, $\mu_p \not\subset K'$, $K(A') \subset K'$ and $\mu_p \subset K''$. 
    \end{enumerate}
\end{example}

Note that this list is not exhaustive. In what follows, we will present some examples that are not covered by Example~\ref{Sha=0}, but still have $\Sha^1_{all}(K, A) = 0$, and some examples with $\Sha^1_{all}(K, A) \neq 0$ that do not come from Wang's counterexample.

A first class of examples comes from the following:
\begin{lemma}\label{shaonezero}
    Let $M \subset G_K$ be the maximal subgroup which acts trivially on $A$, so $M$ fixes $L=K(A)$ and $A^M = A$. If the Sylow-$p$ subgroup of $G_K/M$ is cyclic, then  $\Sha^1_{all}(K, A) = 0$. 
\end{lemma}
\begin{proof}
    Let $0 \neq f \in H^1(G_K, A)$. We will prove that $f \notin \Sha^1_{all}(K, A)$. 

    On the one hand, if $f$ does not inflate from $H^1(G_K/M, A)$, then the inflation-restriction sequence $$0 \to H^1(G_K/M, A) \to H^1(G_K, A) \to H^1(M, A)^{G_K/M}$$ implies that $f$ restricts to a non-zero element of $H^1(M, A)^{G_K/M} = \Hom_{G_K/M}(M, A)$. Thus, $f$ cuts out the non-trivial field extension $L_f/L$ in the diagram below. Note $L_{f}/K$ is a finite Galois extension. Let ${\mathfrak q}$ be any prime of $K$ whose Frobenius in $\Gal(L_{f}/K)$ is a non-trivial element of $\Gal(L_{f}/L)$. Then  $f|_{G_{K_{\mathfrak q}}} \neq 0$ so $f \notin \Sha^1_{all}(K,A)$. This part of the proof does not use the assumption of the cyclicity of the Sylow-$p$ subgroup.
    \[
        \begin{tikzcd} [arrows=dash]
        \overline{K} \arrow[dd] \arrow[dr] \arrow[dd, bend right = 30, "M" swap] \arrow[dddd, bend right = 60, "G_K", swap] & \\ & L_{f} \arrow[dl] \\ L=K(A)\arrow[d] \arrow[dd, bend left = 45, "G_K/M"] \arrow[d, bend right = 30, "Syl_p" swap] & \\ E \arrow[d] & \\K &
        \end{tikzcd}
    \]
    
    On the other hand, if $f$ inflates from $H^1(G_K/M, A)$, then by abuse of notation, let $0 \neq f \in H^1(G_K/M, A)$. Recall we are assuming $A$ is a finite cardinality $\Z_p[G_K]$-module and let $Syl_p$ be a Sylow-$p$ subgroup of $G_K/M$ with fixed field $E$. 
    As $\#A=p^m$ for some $m>0$, the restriction map $H^1(G_K/M,A) \to H^1(Syl_p,A)$ is injective. As $Syl_p = \langle \sigma \rangle$ is cyclic, we can choose a prime $\q_0$ of $E$ whose Frobenius is conjugate to $\sigma$. Then $Syl_p$ is a quotient of the local Galois group $G_{E_{\q_0}}$, so $f|_{G_{E_{\q_0}}}$ inflates from the non-zero element $f|_{Syl_p}$ and is thus non-zero itself. Choosing $\q$ of $K$ below $\q_0$, we obtain
    $f \notin \Sha^1_{all}(K, A)$. 

    Combining these two cases, it follows that $\Sha^1_{all}(K, A) = 0$. 
\end{proof}

\begin{corollary}\label{primetop-sha=0}
    If $\Gal(L/K)$ has order prime to $p$, then $\Sha^1_{all}(K, A) = 0$. Moreover, if one also has $pA = 0$, then $\Sha^1_{all}(K, A') = 0$. 
\end{corollary}
\begin{proof}
    If $\Gamma = \Gal(L/K)$ has order prime to $p$, we are in the case of the proof of Lemma~\ref{shaonezero} where the restriction map 
    $H^1(G_K,A) \to H^1(M,A)^{G_K/M}$ is injective, so
    we get $\Sha^1_{all}(K, A) = 0$. 

    Let $L' = K(A')$. Since $pA=0$, we see that $A' = \Hom(A, \mu_p)$, so we have $L' \subseteq L(\mu_p)$. Thus, $\Gal(L'/K)$ has order prime to $p$. By the first part, we get $\Sha^1_{all}(K, A') = 0$.
\end{proof}

Another class of examples comes from the following:
\begin{lemma} \label{shaA'=0}
    Let $A$ be an $\F_p[G_K]$-module. If $\mu_p \not\subset K(A)$, then $\Sha^1_{all}(K, A') = 0$. 
\end{lemma}
\begin{proof}
    Observe that $A' = \Hom(A, \mu_p)$, since $A$ is an $\F_p[G_K]$-module. Then $K(A') = K(A, \mu_p)$. Let $G = \Gal(K(A')/K)$ and $N = \Gal(K(A')/K(A))$, and note that $N$ has order prime to $p$. Then in the inflation-restriction sequence $$0 \to H^1(G/N, (A')^N) \to H^1(G, A') \to H^1(N, A),$$ $H^1(N, A) = 0$. Moreover, $N$ fixes $A$ and $N$ acts on $\mu_p$ via the non-trivial cyclotomic character, so $(A')^N = 0$, which implies that $H^1(G/N, (A')^N) = 0$. Thus $H^1(G, A') = 0$ and following the proof of the first part of Lemma~\ref{shaonezero}, 
    $\Sha^1_{all}(K, A') = 0$, as desired.  
\end{proof}

We finish this section by observing that classifying the $\Z_p[G_K]$-modules $A$ with $\Sha^1_{all}(K, A) = 0$ is not a trivial task. In particular, we give examples of fields $K$ and modules $A$ that don't fall into any of the categories enumerated above, but still have $\Sha^1_{all}(K, A) = 0$. 
\begin{example}
    Let $p = 3$, $K = \Q$ and let $L$ be the compositum of the splitting fields of $x^3 - x^2 - 26x-41$ and $x^3 - x^2 - 32x + 79$. Then $L/\Q$ is Galois with Galois group isomorphic to $\Z/3\Z \times \Z/3\Z$. Let $\Gamma = \Gal(L/\Q) = \langle \sigma, \tau \rangle$. Let $A$ be the $\F_p[G_\Q]$-module defined by 
    \begin{align*}
        \rho: G_\Q \twoheadrightarrow \Gal(L/\Q) &\to \GL_3(\F_p) \\
        \sigma &\mapsto I + E_{1,2} \\
        \tau &\mapsto I + E_{1, 3}
    \end{align*}
    Note that the field cut out by $A$ is exactly $L$, so $\Q(A) = L$. Moreover, $A$ is not a trivial $G_\Q$-module, $A \neq \mu_{p^n}$, $A$ is not a simple $G_\Q$-module, $\#\Gal(L/\Q) = 9 \neq 3 = lcm\{\#\Gal(L_\q/\Q_\q)\}$, and the Sylow-$3$ subgroup of $\Gal(L/\Q)$ is not cyclic, so this example has a different flavor than the ones in Example~\ref{Sha=0}, Lemma~\ref{shaonezero} and Lemma~\ref{shaA'=0}. 

    We now show that $\Sha^1_{all}(\Q, A) = 0$. To this end, assume there exists $0 \neq [f] \in \Sha^1_{all}(\Q, A)$. This means that $0 \neq [f] \in H^1(\Gamma, A)$ such that the restriction $[f]|_{\Gamma_\q} = 0$, for all primes $\q$. Here $\Gamma_\q \subset \Gamma$ denotes the decomposition subgroup at the prime $\q$. By Chebotarev's theorem, every non-trivial cyclic subgroup of $\Gamma$ occurs as a decomposition subgroup at some prime, so we need $[f]|_{\langle g \rangle} = 0$, for all cyclic subgroups $\langle g \rangle \subset \Gamma$. Let $g = \sigma^i \tau^j \neq 1$. As $f$ is a $1$-cocycle in $H^1(\Gamma, A)$, we have 
    \[f(g) = if(\sigma) + jf(\tau) + \text{ terms in } \F_pe_1.\]
    The assumption $[f]|_{\langle g \rangle} = 0$ implies that $if(\sigma) + jf(\tau) \in \F_pe_1$, for all $(i, j) \neq (0, 0)$. In particular, this implies that $f$ must satisfy $f(\sigma) = (a, 0, 0)$, $f(\tau) = (b, 0, 0)$, for some $a, b \in \F_p$. However, these are exactly the 1-coboundaries in $H^1(\Gamma, A)$, so $[f] = 0$ in $H^1(\Gamma, A)$. It follows that our assumption was false and $\Sha^1_{all}(\Q, A) = 0$. 
\end{example}
\begin{example} Let $K$ be any number field.
    Let $G$ be any non-cyclic $p$-group with $Z$ any central subgroup of order $p$. We now exhibit an $\F_p[G]$-module $I$ with $H^1(G,I) \cong \F_p$ (spanned, say, by $f$) such that the restriction map $H^1(G,I) \to H^1(Z,I)$ is injective. Then using the Scholz-Reichardt theorem we can realize $G$ as $\Gal(L/K)$ and letting $\q$ be a prime of $K$ whose Frobenius spans $Z$, we will have $f\mid_{G_\q} \neq 0$ so $\Sha^1_{all}(K, I) = 0$. Note this example does not fall in any of the cases of Example~\ref{Sha=0}, Lemma~\ref{shaonezero} and Lemma~\ref{shaA'=0}. 

Let $I$ be the augmentation ideal of $\F_p[G]$ and consider the exact sequence
$$0 \to I \to \F_p[G] \to \F_p \to 0.$$ Take $G$-cohomology to get
$$0 \to I^G \to \F_p[G]^G \to \F_p \to H^1(G,I) \to 0.$$
One sees $\F_p[G]^G =\lambda\cdot(\sum_{g\in G}g)$ and subtracting $0=\lambda |G|\cdot e$, this element also spans $I^G$, so $\dim H^1(G,I)=1$. 
One sees $\F_p[G]^Z$ is spanned by the cosets of $Z$ in $G$, namely elements of the form 
$\sum_{g\in G/Z} \lambda_g \left( \sum_{z\in Z} gz\right)$. Note that if we subtract $0=p\cdot e$ from each inner sum,  we get an element of $I$ so $I^Z=\F_p[G]^Z$. This common $\F_p[G/Z]$-module is just the regular representation of $G/Z$ so $H^1(G/Z,I^Z)=0$ and the inflation-restriction sequence 
$$0\to H^1(G/Z,I^Z) \to H^1(G,I) \to H^1(Z,I)$$ 
becomes
$$0\to 0 \to H^1(G,I) \to H^1(Z,I),$$ 
giving our desired injection.
\end{example}
\begin{example} \label{Shanot0} Let $p$ be any odd prime and $K$ be any number field. We give an example of a $\F_p[G_K]$-module $M$ for which $\Sha^1_{all}(K, M) \neq 0$. 
To the best of our knowledge, this is the first
example of $\Sha^1_{all}(K,A)\neq 0$ when $p$ is odd, and were surprised not to find any. Therefore, to our knowledge, this is the first example with $p$ odd for which $\Sha^1_{all}(K, A) \neq 0$.

    Let $B \in \GL_p(\F_p)$ be a permutation matrix corresponding to a $p$-cycle. Using the binomial theorem, $(I+B)^p=I+B^p=I+I$. Thus  
    $C:=(I+B)^{p-1}$ is trivial or has order $p$. Using that $1\leq k\leq p-1$, we see that exactly one of the matrices in the binomial expansion of $C$ has non-zero $(1,2)$-entry. Thus $C\neq I$ and its order is $p$. Clearly $B$ and $C$ commute and we show $G:=\langle B,C\rangle \simeq (\F_p)^2$: suppose $C=B^r$ for some $1\leq r\leq p-1$ and observe the $(1,1)$-entry of $C$ is $1$ while that of $B^r$ is $0$.
    
    Let $M=(\F_p)^p$ with the natural action of $G$. We show that for any proper subgroup $Z \leq G$ that the restriction map $H^1(G,M) \to H^1(Z,M)$ is the zero map. As one can realize $G$ as $\Gal(L/K)$ where the decomposition groups of all ramified primes are cyclic,  $\Sha^1_{all}(K, M) \neq 0$.
   
    Since $B$ is a $p$-cycle permutation matrix, we see 
    that the action of $\langle B \rangle $ on $M$ is via the regular representation of $\langle B \rangle $. Also, $M^B$ is just the scalar column vectors in $M$, so
    the inflation-restriction sequence
    $$0 \to H^1(G/\langle B\rangle,M^{\langle B \rangle }) \to H^1(G,M) \to 
    H^1(\langle B\rangle,M)$$
    has trivial last term and $H^1(G/\langle B\rangle,M^{\langle B \rangle }) \simeq \Hom(\F_p,\F_p)$ so $\dim H^1(G,M)=1$.

    Now let $N$ be a proper (necessarily) normal subgroup of $G$ and note $0 \subsetneq M^N \subsetneq M$ so 
    $0<\dim M^N<p$. Thus $M^N$ is not the regular representation of $G/N\simeq \F_p$ and $H^1(G/N,M^N) \neq 0$ (see \cite[Theorem 6]{MR219512}). Thus in
    the inflation-restriction sequence
    $$0 \to H^1(G/N,M^N) \to H^1(G,M) \to 
    H^1(N,M)$$
    the first map must be an isomorphism of one-dimensional spaces and the restriction map to $N$ must be trivial, as desired.
\end{example}

\section{Some Lemmas}\label{SomeLemmas}
The running hypotheses and notes for this section are:
\begin{itemize}[leftmargin=*]
    \item $p$ is a prime. 
    \item $A$ is a finite cardinality $\Z_p[G_K]$-module. 
    \item $T$ is a finite set of tame primes of $K$ that contains all the primes that ramify in $L/K$.
\end{itemize}

Lemma~\ref{lemma:ShaInj} is used throughout the paper, including in the examples in \S\ref{examples}. 
Lemmas~\ref{lemma:coh} and~\ref{lemma:equiv} play a crucial role in the proofs of Proposition~\ref{thm: main} and  
Theorem~\ref{thm:main2}.

Since $T$ contains all the primes that ramify in $L/K$ and $L$ is the splitting field of $A$, the group $G_{L, \tilde{T}}$ acts trivially on $A$. Therefore,
$
 A \simeq \oplus_i\Z/p^{n_i}\Z
$
as a $\Z_p[G_{L, \tilde{T}}]$-module, for some $n_i \geq 1$. 
\begin{lemma}\label{lemma:ShaInj}
With the above notation,
   $\Sha^2_{\tilde{T}, p}(L,A) \hookrightarrow \Sha^2_{\tilde{T}}(L,A)$. 
\end{lemma}
\begin{proof}
    Let $J$ be the kernel of $G_{L,{\tilde{T}}} \twoheadrightarrow G_{L,{\tilde{T}}}(p)$. Any 
    $0 \neq \alpha \in H^1(J, A)^{G_{L,{\tilde{T}}}(p)}$ gives rise to a non-trivial extension of $L_{\tilde{T}}(p)$ lying within $L_{\tilde{T}}$ and Galois over $L$. This contradicts the maximality of $L_{\tilde{T}}(p)$, so $H^1(J,A)^{G_{L,{\tilde{T}}}(p)}=0$. From the inflation-restriction sequence we have
    $H^2(G_{L,{\tilde{T}}}(p),A) \hookrightarrow H^2(G_{L,{\tilde{T}}},A)$. It is easy to see elements of  $\Sha^2_{{\tilde{T}}, p}(L,A)$ inflate to elements of
     $\Sha^2_{\tilde{T}}(L,A)$.
\end{proof}

\begin{lemma}\label{lemma:coh} 
    If $(\#\Gamma, p) = 1$, then:
    \begin{enumerate}[(i)]
        \item $H^i(G_{K, T}, A) \simeq H^i(G_{L, \tilde{T}}, A)^\Gamma$, for all $i \geq 1$.
        \item Let $\q \in T$. Then $H^i(G_{K_\q}, A) \simeq \left(\prod_{{\tilde{\q}} | \q} H^i(G_{L_{\tilde{\q}}}, A) \right)^\Gamma$, for all $i \geq 1$. 
        \item Let $\q \in T$. Then $H^1_{nr}(G_{K_\q}, A) \simeq \left( \prod_{{\tilde{\q}} | \q} H^1_{nr}(G_{L_{\tilde{\q}}}, A) \right)^\Gamma$.
    \end{enumerate}
\end{lemma}
\begin{proof}
    {\it (i)}: As $(\#\Gamma,\#A)=1$,  we have for $m>0$ and $n \geq 0$ that  $H^m(\Gamma, H^n(G_{L, \tilde{T}}, A)) = 0$. Combining this with \cite[Lemma 2.1.2]{MR2392026} (the spectral sequence degenerates), we  conclude that $H^i(G_{K, T}, A) \simeq H^i(G_{L, \tilde{T}}, A)^\Gamma$, for all $i \geq 1$.  

    {\it (ii)}: Since $L/K$ is Galois, the decomposition groups $\Gal(L_{\tilde{\q}}/K_\q)$ are canonically isomorphic up to conjugacy for all ${\tilde{\q}} | \q$. Abusing notation, let $\Gamma_\q = \Gal(L_{\tilde{\q}}/K_\q) \subset \Gamma$. By \cite[Lemma 2.1.2]{MR2392026} again, $H^i(G_{K_\q}, A) \simeq (H^i(G_{L_{\tilde{\q}}}, A))^{\Gamma_\q}$ for $i \geq 1$. As $\Gamma$ acts transitively on the  $\# \Gamma/\Gamma_\q$ primes ${\tilde{\q}}$ above $\q$ in $\tilde{T}$,  the result follows. 

    {\it (iii)}:  Recall $\mathcal{I}_\q(K)$ and $\mathcal{I}_{\tilde{\q}}(L)$ are the inertia subgroups of the absolute local Galois groups $G_{K_\q}$ and $G_{L_{\tilde{q}}}$.
    Set ${\mathcal I}(\Gamma_\q)\subset \Gamma_\q := \Gal(L_{\tilde{\q}}/K_\q) \subset \Gamma$ to be the inertia subgroup of $\Gamma = \Gal(L/K)$ at $\q$ (as  $L/K$ is Galois, all the inertia subgroups are canonically isomorphic up to conjugacy). 
    By part {\it (ii)} and \cite[Lemma 2.1.2]{MR2392026}, we see that $H^1(G_{K_\q}, A) \simeq H^1(G_{L_{\tilde{\q}}}, A)^{\Gamma_\q}$ 
    and $H^1(\mathcal{I}_\q(K), A) \simeq H^1(\mathcal{I}_{\tilde{\q}}(L), A)^{{\mathcal I}(\Gamma_\q)}$. 
    Since ${\mathcal I}(\Gamma_\q) \subset \Gamma_\q$ we have   
    $H^1(\mathcal{I}_{\tilde{\q}}(L), A)^{\Gamma_\q} \subset H^1(\mathcal{I}_{\tilde{\q}}(L), A)^{{\mathcal I}(\Gamma_\q)}$. 
    Using the fact that there are $\#\Gamma/\Gamma_\q$ primes ${\tilde{\q}}$ above $\q$ in $\tilde{T}$, we observe that 
    $$\left(\prod_{{\tilde{\q}} | \q} H^1(\mathcal{I}_{\tilde{\q}}(L), A)\right)^\Gamma \simeq H^1(\mathcal{I}_{\tilde{\q}}(L), A)^{\Gamma_\q} 
    \hookrightarrow H^1(\mathcal{I}_{\tilde{\q}}(L), A)^{{\mathcal I}(\Gamma_\q)}
    \simeq H^1(\mathcal{I}_\q(K), A).$$
    Combining this with the fact that $H^1(G_{K_\q}, A) \simeq (\prod_{{\tilde{\q}} | \q} H^1(G_{L_{\tilde{\q}}}, A))^\Gamma$, we obtain the following commutative diagram: 
    \[
    \begin{tikzcd}
        H^1_{nr}(G_{K_\q}, A) \arrow[d, hook] & \left( \prod_{{\tilde{\q}} | \q} H^1_{nr}(G_{L_{\tilde{\q}}}, A) \right)^\Gamma \arrow[d, hook] \\
        H^1(G_{K_\q}, A) \arrow[d] \arrow[r, "\simeq"] & \left( \prod_{{\tilde{\q}} | \q} H^1(G_{L_{\tilde{\q}}}, A) \right)^\Gamma \arrow[d] \\
        H^1(\mathcal{I}_\q(K), A) & \left( \prod_{{\tilde{\q}} | \q} H^1(\mathcal{I}_{\tilde{\q}}(L), A) \right)^\Gamma \arrow[l, hook']
    \end{tikzcd}
    \]
    Note that the exactness of the right column follows from the left exactness of taking $\Gamma$-invariants. A  diagram chase shows that $H^1_{nr}(G_{K_\q}, A) \simeq (\prod_{{\tilde{\q}} | \q} H^1_{nr}(G_{L_{\tilde{\q}}}, A))^\Gamma$. 
\end{proof}

\begin{lemma}\label{lemma:equiv}
    Let $A$ be a finite cardinality $\Z_p[G_K]$-module and let $T$ be a finite set of tame primes of $K$ that contains all the ramified primes of $L/K$. If $p = 2$, assume, moreover, that $4A = 0$. 
    \begin{enumerate}[(i)]
        \item There are $\Gamma$-equivariant  inclusions $\Sha_{\tilde{T}, p}^2(L, A) \hookrightarrow \Sha_{\tilde{T}}^2(L, A) \hookrightarrow \RusB_{\tilde{T}}(L, A)$. 
        \item $\Sha^2_{\tilde{T}, p}(L, A)^\Gamma \simeq \Sha^2(L_{\tilde{T}}(p)/K, A)$.
        \item $\Sha_{\tilde{T}}^2(L, A)^\Gamma \simeq \Sha_T^2(K, A)$.
        \item $\RusB_{\tilde{T}}(L, A)^\Gamma \simeq \RusB_T(K, A)$.
    \end{enumerate}
\end{lemma}
\begin{proof}
    {\it (i)}: Since $A$ is a trivial $\Z_p[G_L]$-module, we can apply Lemma~\ref{lemma:ShaInj} to get an injection $\Sha^2_{\tilde{T}, p}(L, A) \hookrightarrow \Sha^2_{\tilde{T}}(L, A)$ induced by the inclusion $H^2(G_{L, \tilde{T}}(p), A) \hookrightarrow H^2(G_{L, \tilde{T}}, A)$. As cohomology respects the conjugation by $\Gamma = \Gal(L/K)$, this inclusion is $\Gamma$-equivariant. The second map follows from Theorem~\ref{thm:Shaisinrusb}{\it (ii)} and Example~\ref{Sha=0}(ii) and (iii), as follows. Since $A$ is a trivial $\Z_p[G_L]$-module, 
    $\displaystyle \Sha^1_{all}(L, A') \simeq \oplus_i \Sha^1_{all}(L, \mu_{p^{n_i}})$.
    When $p$ is odd, by Example~\ref{Sha=0}(ii), $\Sha^1_{all}(L, A') = 0$. When $p = 2$ and $4A = 0$, we get $4A' = 0$, so Example~\ref{Sha=0}(iii) applies and $\Sha^1_{all}(L, A') = 0$. In either case
    we apply Theorem~\ref{thm:Shaisinrusb} to get a $\Gamma$-equivariant injection $\Sha^2_{\tilde{T}}(L, A) \hookrightarrow \RusB_{\tilde{T}}(L, A)$.
    
    {\it (ii)}: 
    We  use \cite[Lemma 2.1.2]{MR2392026} to observe that $H^2(G_{L,\tilde{T}}(p)\rtimes \Gamma, A) \simeq H^2(G_{L, \tilde{T}}(p), A)^\Gamma$. Using Lemma \ref{lemma:coh}{\it (ii)} we obtain:
    \begin{align*}
        \Sha^2(L_{\tilde{T}}(p)/K, A) &= \ker\left(H^2(G_{L, \tilde{T}}(p) \rtimes \Gamma, A) \to \prod_{\q \in T}H^2(G_{K_\q}, A) \right) \\
        &\simeq \ker\left(H^2(G_{L, \tilde{T}}(p), A)^\Gamma \to \prod_{\q \in T}\left( \prod_{{\tilde{\q}} | \q} H^2(G_{L_{\tilde{\q}}}, A) \right)^\Gamma \right) \\
        & \simeq \ker \left( H^2(G_{L, \tilde{T}}(p), A) \to \prod_{{\tilde{\q}} \in \tilde{T}} H^2(G_{L_{\tilde{\q}}}, A)\right)^\Gamma \\
        &= \Sha^2_{\tilde{T}, p}(L, A)^\Gamma,
    \end{align*}
where the second isomorphism follows from left exactness of taking $\Gamma$-invariants.

    {\it (iii)}: We  use Lemma~\ref{lemma:coh} parts {\it (i)} and {\it (ii)}:  
\begin{align*}
    \Sha_T^2(K, A) &= \ker \left( H^2(G_{K, T}, A) \to \prod_{\q \in T} H^2(G_{K_\q}, A) \right) \\
    & \simeq \ker \left( H^2(G_{L, \tilde{T}}, A)^\Gamma \to \prod_{\q \in T} \left(\prod_{{\tilde{\q}} | \q} H^2(G_{L_{\tilde{\q}}}, A)\right)^\Gamma \right) \\
    & \simeq \ker \left(H^2(G_{L, \tilde{T}}, A) \to \prod_{{\tilde{\q}} \in \tilde{T}} H^2(G_{L_{\tilde{\q}}}, A) \right)^\Gamma \\
    & = \Sha_{\tilde{T}}^2(L, A)^\Gamma,
\end{align*}
where the second isomorphism follows from left exactness of taking $\Gamma$-invariants.

{\it (iv)}: We use Lemma~\ref{lemma:coh}, parts {\it (i)}, {\it (ii)}, and {\it (iii)}: 
\begin{align*}
    \RusB_T(K, A) &= \coker \left( \prod_{\q \in T} H^1(G_{K_\q}, A) \times \prod_{\q \notin T} H^1_{nr}(G_{K_\q}, A) \to H^1(G_K, A')^\vee\right) \\
    & \simeq \coker \left( \prod_{\q \in T} \left(\prod_{{\tilde{\q}} | \q} H^1(G_{L_{\tilde{\q}}}, A)\right)^\Gamma \times \prod_{\q \notin T}\left(\prod_{{\tilde{\q}} | \q} H^1_{nr}(G_{L_{\tilde{\q}}}, A)\right)^\Gamma \to (H^1(G_L, A')^\vee)^\Gamma \right) \\
    & \simeq \coker \left( \prod_{{\tilde{\q}} \in \tilde{T}} H^1(G_{L_{\tilde{\q}}}, A) \times \prod_{{\tilde{\q}} \notin \tilde{T}} H^1_{nr}(G_{L_{\tilde{\q}}}, A) \to H^1(G_L, A')^\vee \right)^\Gamma \\
    &= \RusB_{\tilde{T}}(L, A)^\Gamma. 
\end{align*}
The second isomorphism uses the fact that since $(\#\Gamma,p)=1$,  taking $\Gamma$-invariants on finite cardinality $\Z_p[\Gamma]$-modules is exact.
\end{proof}

\section{Proof of Main Results}\label{MainResults}
The running hypotheses and notes for this section are:
\begin{itemize}[leftmargin=*]
    \item $p$ is an odd prime. This is required to invoke the main result of
    \cite{MR4359923}. See \S \ref{p=2} for $p=2$. 
    \item $A$ is a finite cardinality $\F_p[G_K]$-module. We  restrict ourselves to this case because the main result of \cite{MR4359923} is currently only known for $\F_p$. 
    \item $S$ is a finite set (possibly empty) of tame primes of $K$; $\tilde{S}$ is the set of places of $L$ above $S$. 
\end{itemize}

Since $G_L$ acts trivially on the $\F_p[G_K]$-module $A$, we can now consider the groups $\Sha^2_{\tilde{S}, p}(L, \F_p)$, $\Sha^2_{\tilde{S}}(L, \F_p)$ and $\RusB_{\tilde{S}}(L, \F_p)$ with the usual inclusions 
\[
\Sha^2_{\tilde{S}, p}(L, \F_p) \hookrightarrow \Sha_{\tilde{S}}^2(L, \F_p) \hookrightarrow \RusB_{\tilde{S}}(L, \F_p).
\]
Use \cite[Theorem A]{MR4359923} to construct a tame set $X'$ of primes of $L$ such that: 
\begin{center}
\begin{tikzcd} \Sha_{\tilde{S}, p}^2(L, \F_p) \arrow[r, hook] & \Sha^2_{\tilde{S}}(L, \F_p) \arrow[r, hook] & \RusB_{\tilde{S}}(L, \F_p) \arrow[d, two heads, "\simeq"] \\
\Sha_{\tilde{S} \cup X^\prime, p}^2(L, \F_p) \arrow[r, hook, "\simeq"] & \Sha_{\tilde{S} \cup X^\prime}^2(L, \F_p) \arrow[r, hook, "\simeq"] & \RusB_{\tilde{S}\cup X^\prime}(L, \F_p). 
\end{tikzcd}
\end{center}
Note that $X'$ is a set of primes of $L$ that is not necessarily $\Gamma$-stable. Let $\tilde{X}$ be the $\Gamma$-orbit of  $X'$.
By \cite[Corollary 1.17]{MR4359923}, $\Sha^2_{\tilde{S} \cup \tilde{X}, p}(L, \F_p) \stackrel{\simeq}{\hookrightarrow} \Sha_{\tilde{S} \cup \tilde{X}}^2(L, \F_p) \stackrel{\simeq}{\hookrightarrow} \RusB_{\tilde{S}\cup \tilde{X}}(L, \F_p)$. We claim that enlarging the set $X'$ to make it $\Gamma$-stable does not decrease the size of the group $\RusB_{\tilde{S}\cup X'}(L, \F_p)$: 
\begin{lemma}\label{RusBpreserved}
    With the above notation, 
    $\RusB_{\tilde{S}}(L, \F_p) \simeq \RusB_{\tilde{S} \cup X^\prime} (L, \F_p) \simeq \RusB_{\tilde{S} \cup \tilde{X}} (L, \F_p)$. 
\end{lemma}
\begin{proof} 
    We work with the duals instead. In the proof of \cite[Prop 2.1]{MR4359923},
     the primes of $X'$ are chosen to split completely in the governing extension 
    $L\left(\zeta_p,\sqrt[p]{V_{\tilde{S}}(L,\F_p)}\right)/L$. 
    Note that as $\tilde{S}$ is $\Gamma$-stable, 
    $L\left(\zeta_p,\sqrt[p]{V_{\tilde{S}}(L,\F_p)}\right)/K$ is Galois.
    Thus the primes of $\Gamma \cdot X'=\tilde{X}$ also split completely in 
    $L\left(\zeta_p,\sqrt[p]{V_{\tilde{S}}(L,\F_p)}\right)/L$, 
    so $V_{\tilde{S}\cup \tilde{X}}(L,\F_p)=V_{\tilde{S}}(L,\F_p)$ as needed.
\end{proof}

We have: 
\begin{center}
\begin{tikzcd} \Sha^2_{\tilde{S}, p}(L, \F_p) \arrow[r, hook] &  \Sha^2_{\tilde{S}}(L, \F_p) \arrow[r, hook] & \RusB_{\tilde{S}}(L, \F_p) 
\arrow[d, two heads, "\simeq"]
 \\
\Sha^2_{\tilde{S} \cup X', p}(L, \F_p) \arrow[r, hook, "\simeq"] & \Sha_{\tilde{S} \cup X^\prime}^2(L, \F_p) \arrow[r, hook, "\simeq"] & \RusB_{\tilde{S}\cup X^\prime}(L, \F_p) \arrow[d, two heads, "\simeq"]\\
\Sha^2_{\tilde{S} \cup \tilde{X}, p}(L, \F_p) \arrow[r, hook, "\simeq"] & \Sha^2_{\tilde{S} \cup \tilde{X}}(L, \F_p) \arrow[r, hook, "\simeq"] & \RusB_{\tilde{S} \cup \tilde{X}}(L, \F_p)
\end{tikzcd}
\end{center}

\begin{proof}[Proof of Proposition \ref{thm: main}]
Let $X$ be the set of primes of $K$ lying below $\tilde{X}$. Recall that $G_L$ acts trivially on $A$. By Lemma~\ref{lemma:equiv}{\it (i)} the following diagram is $\Gamma$-equivariant: 
\begin{equation}
 \label{eqn:equivariant}
 \Sha^2_{\tilde{S}, p}(L, A) \hookrightarrow \Sha^2_{\tilde{S}}(L, A) \hookrightarrow \RusB_{\tilde{S}}(L, A) \stackrel{\simeq}{\twoheadrightarrow} \RusB_{\tilde{S} \cup \tilde{X}}(L, A)  \stackrel{\simeq}{\hookleftarrow} \Sha^2_{\tilde{S}\cup \tilde{X}}(L, A)
 \stackrel{\simeq}{\hookleftarrow} \Sha^2_{\tilde{S} \cup \tilde{X}, p}(L, A).
\end{equation}
 We can take $\Gamma$-invariants of (\ref{eqn:equivariant}) to obtain
\[
\begin{tikzcd}
    \Sha^2_{\tilde{S}, p}(L, A)^\Gamma \arrow[r, hook] & \Sha^2_{\tilde{S}}(L,A)^\Gamma \arrow[r, hook] & \RusB_{\tilde{S}}(L, A)^\Gamma \arrow[d, two heads, "\simeq"] \\ \Sha^2_{\tilde{S} \cup \tilde{X}, p}(L, A)^\Gamma \arrow[r, hook, "\simeq"] & \Sha^2_{\tilde{S} \cup \tilde{X}}(L, A)^\Gamma \arrow[r, hook, "\simeq"] & \RusB_{\tilde{S} \cup \tilde{X}}(L, A)^\Gamma
\end{tikzcd}
\]
 proving Proposition~\ref{thm: main}. 

\end{proof}

\begin{proof}[Proof of Theorem \ref{thm:main2}]
To prove Theorem~\ref{thm:main2}, recall that $(\#\Gamma, p) = 1$, so by Corollary~\ref{primetop-sha=0} we know that $\Sha^1_{all}(K, A') = 0$, and so by Theorem~\ref{thm:Shaisinrusb} we have an injection $\Sha^2_S(K, A) \hookrightarrow \RusB_S(K, A)$. Now, take $\Gamma$-invariants of (\ref{eqn:equivariant}) and use  Lemma~\ref{lemma:equiv} to see
 \[
\begin{tikzcd}
    \Sha^2(L_{\tilde{S}}(p)/K, A) \arrow[r, hook] & \Sha^2_{S}(K,A) \arrow[r, hook] & \RusB_{S}(K, A) \arrow[d, two heads, "\simeq"] \\ \Sha^2(L_{\tilde{S}\cup \tilde{X}}(p)/K, A)\arrow[r, hook, "\simeq"] & \Sha^2_{S \cup X}(K, A) \arrow[r, hook, "\simeq"] & \RusB_{S \cup X}(K, A)
\end{tikzcd}
\]
\end{proof}

{Corollary~\ref{cor:isoarepreserved} below shows that once an isomorphism $\Sha^2(L_{\tilde{S}}(p)/K, A) \stackrel{\simeq}{\hookrightarrow} \RusB_S(K, A)$ is established, increasing the set $S$ by \textit{any} set $Y$ (tame or wild) preserves the isomorphism $\Sha^2(L_{\tilde{S} \cup \tilde{Y}}(p)/K, A) \stackrel{\simeq}{\hookrightarrow}\RusB_{S \cup Y}(K, A)$.}

\begin{corollary}\label{cor:isoarepreserved}
    If the composite injection $\Sha^2(L_{\tilde{S}}(p)/K, A) \hookrightarrow \Sha^2_S(K, A) \hookrightarrow \RusB_S(K,A)$ is an isomorphism, then for every set $Y$ of primes of $K$ we have $$\Sha^2(L_{\tilde{S} \cup \tilde{Y}}(p)/K, A) \stackrel{\simeq}{\hookrightarrow} \Sha^2_{S\cup Y}(K, A) \stackrel{\simeq}{\hookrightarrow} \RusB_{S\cup Y}(K, A).$$
\end{corollary}
\begin{proof}
    This proof is inspired by the proofs of \cite[Lemma 1.16]{MR4359923} and \cite[Lemma 8.4]{MR4896734}. 
    By the inflation-restriction exact sequence, one has: 
    $$H^1(\Gal(L_{\tilde{S} \cup \tilde{Y}}(p)/L_{\tilde{S}}(p)), A)^{G_{L, \tilde{S}}(p) \rtimes \Gamma} \to H^2(G_{L, \tilde{S}}(p)\rtimes \Gamma, A) \to H^2(G_{L, \tilde{S} \cup \tilde{Y}}(p) \rtimes \Gamma, A).$$
    Note that the image of $H^1(\Gal(L_{\tilde{S} \cup \tilde{Y}}(p)/L_{\tilde{S}}(p)), A)^{G_{L, \tilde{S}}(p) \rtimes \Gamma}$ in $H^2(G_{L, \tilde{S}}(p) \rtimes \Gamma, A)$ lies in $\Sha^2(L_{\tilde{S}}(p)/K, A)$, whose image lies in $\Sha^2(L_{\tilde{S} \cup \tilde{Y}}(p)/K, A)$, so we have the following commutative diagram: 
    \begin{center}
\begin{tikzcd} [row sep=normal, column sep = small]
 H^1(\Gal(L_{\tilde{S} \cup \tilde{Y}}(p)/L_{\tilde{S}}(p)), A)^{G_{L, \tilde{S}}(p) \rtimes \Gamma} \arrow[r]  & \Sha^2(L_{\tilde{S}}(p)/K, A) \arrow[r] \arrow[d, hook] & \Sha^2(L_{\tilde{S}\cup \tilde{Y}}(p)/K, A) \arrow[d, hook] \\
 & \RusB_{S}(K, A) \arrow[r, two heads] & \RusB_{S \cup Y}(K, A)
\end{tikzcd}
\end{center}
We immediately obtain that if $\Sha^2(L_{\tilde{S}}(p)/K, A) \hookrightarrow \RusB_S(K, A)$ is an isomorphism, then so is $\Sha^2(L_{\tilde{S}\cup \tilde{Y}}(p)/K, A) \hookrightarrow \RusB_{S \cup Y}(K, A)$. 

\end{proof}

\section{The case $p = 2$}\label{p=2}
As is usual, the case $p=2$ presents additional technical difficulties. We retain the (stronger) hypotheses of Theorem~\ref{thm:main2} (and not of Proposition~\ref{thm: main}):
\begin{itemize}[leftmargin=*]
    \item $p=2$,
    \item $A$ is a finite cardinality  $\F_2[G_K]$-module  
    with splitting field $L:=K(A)$, and $\Gamma = \Gal(L/K)$,
    \item $(\#\Gamma,2)=1$,
    \item $S$ contains all ramified places of $L/K$.
\end{itemize}

The main result of this 
paper, Theorem~\ref{thm:main2}, is a generalization of \cite[Theorem A]{MR4359923}. Both of these results hold only for odd primes $p$, so it is natural to ask what happens when $p = 2$. In \cite{MR4359923}, the case $p = 2$ involves an exceptional situation. 
\begin{definition}
    For a field $K$ and a tame set $S$ of primes of $K$, the pair $(K,S)$  is called exceptional if $p = 2$ and the following conditions hold simultaneously: 
    \begin{enumerate}[label=(\alph*)]
        \item $\zeta_4 \notin K$,
        \item $\mathcal{O}_K^\times \cap -4K^4 \neq \emptyset$,
        \item $S$ contains no real place, and for every $\q \in S$ one has $\zeta_4 \in K_\q$. 
    \end{enumerate}
\end{definition}

It is shown in \cite[Theorem B]{MR4359923} that if $S$ is a set of tame primes of $K$ and the situation is exceptional, then there exist infinitely many finite tame sets $X$ such that 
$$d_2 \RusB_S(K, \F_2) - 1  \leq d_2 \Sha^2_{S \cup X, 2}(K, \F_2) = d_2 \Sha^2_{S \cup X}(K, \F_2) = d_2 \RusB_{S \cup X}(K, \F_2) \leq d_2 \RusB_S(K, \F_2).$$
If the situation is not exceptional, then the conclusion of \cite[Theorem A]{MR4359923} holds in this case too. 

In this section we prove a generalization of this result:
\begin{theorem}
    \begin{enumerate}[label=(\roman*)]
        \item If $(K,S)$ is not exceptional, or if the trivial module $\F_2$ does not occur in the decomposition of $A$ as a direct sum of irreducible $\F_2[\Gamma]$-modules, then there exist infinitely many tame sets $X$ of $K$ such that $$\Sha^2(L_{\tilde{S}\cup \tilde{X}}(p)/K, A) \stackrel{\simeq}{\hookrightarrow} \Sha^2_{S \cup X}(K, A) \stackrel{\simeq}{\hookrightarrow} \RusB_{S\cup X}(K, A) \stackrel{\simeq}{\twoheadleftarrow} \RusB_S(K, A).$$
        \item Let $\F_2$ be the trivial $\F_2[\Gamma]$-module, so $\F_2\otimes A \simeq A$ as $\F_2[\Gamma]$-modules.
        If $(K,S)$ is exceptional, then there exist infinitely many finite tame sets $X$ such that either
        $$\Sha^2(L_{\tilde{S}\cup \tilde{X}}(p)/K, A) \stackrel{\simeq}{\hookrightarrow} \Sha^2_{S \cup X}(K, A) \stackrel{\simeq}{\hookrightarrow} \RusB_{S\cup X}(K, A) \stackrel{\simeq}{\twoheadleftarrow} \RusB_S(K, A)$$
        or
        $$(\F_2 \otimes A)^\Gamma \oplus \Sha^2(L_{\tilde{S}\cup \tilde{X}}(p)/K, A) \stackrel{\simeq}{\hookrightarrow} A^\Gamma \oplus \Sha^2_{S \cup X}(K, A) \stackrel{\simeq}{\hookrightarrow}  A^\Gamma \oplus \RusB_{S \cup X}(K, A) \stackrel{\simeq}{\twoheadleftarrow} \RusB_S(K, A).$$
    \end{enumerate}
\end{theorem}

\begin{proof}
    The proof follows the same steps as the proof of Theorem~\ref{thm:main2}, with the caveat that we need to use \cite[Theorem B]{MR4359923} instead of \cite[Theorem A]{MR4359923} in the construction of the tame set $X'$ of primes of $L$. Note that the proof of Theorem~\ref{thm:main2} only used the statement of \cite[Theorem A]{MR4359923}, not needing the details that go into proving this result. In contrast, in the exceptional case of this result, we need to be careful when choosing the set $X'$.
    
    We first claim that the situation is exceptional for $(K,S)$ if and only if it is exceptional for $(L,\tilde{S})$. 
    Since $\Gamma = \Gal(L/K)$ has odd order,  $\zeta_4 \notin K$ (resp. $\zeta_4 \in K_\q$ for $\q \in S$) if and only if $\zeta_4 \notin L$ (resp. $\zeta_4 \in L_{\tilde{\q}}$, for ${\tilde{\q}} | \q$). Similarly, $S$ contains no real place if and only if $\tilde{S}$ contains no real place. Finally, 
    if $\mathcal{O}_K^\times \cap -4K^4 \neq \emptyset$, then clearly $\mathcal{O}_L^\times \cap -4L^4 \neq \emptyset$. On the other hand, if $ -4a^4 \in \mathcal{O}_L^\times \cap -4L^4$ (for $a \in L$), then taking 
    norms, we obtain, using that $\#\Gamma = 2k+1$,

    $$\displaystyle N_{L/K}(-4a^4)  = -4 \left(2^k N_{L/K}(a)\right)^{4} \in \mathcal{O}_K^\times \cap -4K^4.$$ 

    We prove parts {\it (i)} and {\it (ii)}  simultaneously. If the situation is not exceptional for $(K,S)$, then it is not exceptional for $(L,\tilde{S})$, so by \cite[Theorem B]{MR4359923}, there exist infinitely many sets $X'$ of primes of $L$ such that
    $$\Sha^2_{\tilde{S} \cup X', 2}(L, \F_2) \stackrel{\simeq}{\hookrightarrow}\Sha^2_{\tilde{S}\cup X'}(L, \F_2) \stackrel{\simeq}{\hookrightarrow} \RusB_{\tilde{S}\cup X'} (L, \F_2) \stackrel{\simeq}{\twoheadleftarrow} \RusB_{\tilde{S}}(L, \F_2).$$ The proof of the first part of {\it (i)} then proceeds exactly as the proof of Theorem~\ref{thm:main2}. We prove the second part of {\it (i)} at the end.

    To finish the proof of {\it {(ii)}}, we can therefore assume that the situation is exceptional for $(K,S)$ and also  $(L,\tilde{S})$. Use \cite[Theorem B]{MR4359923} to construct infinitely many sets $X'$ such that 
    $$d_2 \RusB_{\tilde{S}}(L, \F_2) - 1 \leq d_2 \Sha^2_{\tilde{S} \cup X', 2}(L, \F_2) = d_2 \Sha^2_{\tilde{S}\cup X'} (L, \F_2) = d_2 \RusB_{\tilde{S} \cup X'}(L, \F_2) \leq d_2 \RusB_{\tilde{S}}(L, \F_2).$$ 
    These sets $X'$ are chosen in very similar ways in both Theorem A and Theorem B of \cite{MR4359923}, with the difference that in the exceptional case one chooses an additional prime $\tilde{\mathfrak{r}}$ that remains inert in $L(i)/L$ and adds it to $X'$: now the pair $(L, \tilde{S} \cup X')$ is not exceptional anymore and the proof of \cite[Theorem B]{MR4359923} follows in the same way as the proof of \cite[Theorem A]{MR4359923}, with
    $d_2 \RusB_{\tilde{S} \cup X'}(L, \F_2)$ equal to either
    $d_2 \RusB_{\tilde{S}}(L, \F_2)-1$ or $d_2 \RusB_{\tilde{S}}(L, \F_2)$.
    
    Our situation imposes an extra difficulty: the group $\Gamma = \Gal(L/K)$ acts on this prime $\tilde{\mathfrak{r}}$, so adding its $\Gamma$-orbit $\tilde{X}$ to our set could potentially make 
    $\RusB_{\tilde{S}\cup \tilde{X}}(L, \F_2)$ smaller than $\RusB_{\tilde{S}\cup X'}(L, \F_2)$. To this end, consider the $\F_2[\Gamma]$-module $\Gal\left(L\left(\sqrt{V_{\tilde{S}}(L, \F_2)}\right)/L\right)$. 
    This can be decomposed as a direct sum of irreducible components, and observe that $\Gal(L(i)/L)$ will correspond to a trivial representation in this decomposition. 
    We can thus choose the prime $\tilde{\mathfrak{r}}$ to remain inert in the trivial copy corresponding to $\Gal(L(i)/L)$ and to have trivial Frobenius in all the other irreducible components. 
    Then the $\Gamma$-orbit of Frobenius at $\tilde{\mathfrak{r}}$ will span a one-dimensional space in
    $\Gal\left(L\left(\sqrt{V_{\tilde{S}}(L, \F_2)}\right)/L\right)$
    and $d_2\RusB_{\tilde{S}\cup \tilde{X}}(L,\F_2)$ will equal
    $d_2 \RusB_{\tilde{S}}(L, \F_2)-1$ or $d_2 \RusB_{\tilde{S}}(L, \F_2)$ as desired.
    It follows that   
    $$d_2 \RusB_{\tilde{S}}(L, \F_2) - 1 \leq d_2 \Sha^2_{\tilde{S} \cup \tilde{X}, 2}(L, \F_2) = d_2 \Sha^2_{\tilde{S}\cup \tilde{X}} (L, \F_2) = d_2 \RusB_{\tilde{S} \cup \tilde{X}}(L, \F_2) \leq d_2 \RusB_{\tilde{S}}(L, \F_2),$$
    so either
    $$\Sha^2_{\tilde{S} \cup \tilde{X}, 2}(L, \F_2) \stackrel{\simeq}{\hookrightarrow} \Sha^2_{\tilde{S} \cup \tilde{X}}(L, \F_2) \stackrel{\simeq}{\hookrightarrow} \RusB_{\tilde{S}\cup \tilde{X}}(L, \F_2) \stackrel{\simeq}{\twoheadleftarrow} \RusB_{\tilde{S}}(L, \F_2) $$
    or
    $$\F_2 \oplus \Sha^2_{\tilde{S} \cup \tilde{X}, 2}(L, \F_2) \stackrel{\simeq}{\hookrightarrow} \F_2 \oplus \Sha^2_{\tilde{S} \cup \tilde{X}}(L, \F_2) \stackrel{\simeq}{\hookrightarrow} \F_2 \oplus \RusB_{\tilde{S} \cup \tilde{X}}(L, \F_2) \stackrel{\simeq}{\twoheadleftarrow} \RusB_{\tilde{S}}(L, \F_2).$$
    As $A$ is an $\F_2[\Gamma]$-module, we have $\RusB_{\tilde{S}}(L, A) \simeq \RusB_{\tilde{S}}(L, \F_2) \otimes A$ (and similarly for $\RusB_{\tilde{S}\cup \tilde{X}}(L, A)$, $\Sha^2_{\tilde{S}\cup \tilde{X}}(L, A)$ and $\Sha^2_{\tilde{S}\cup \tilde{X}, p}(L, A)$). So {either}
    $$\Sha^2_{\tilde{S} \cup \tilde{X}, 2}(L, A) \stackrel{\simeq}{\hookrightarrow} \Sha^2_{\tilde{S} \cup \tilde{X}}(L, A) \stackrel{\simeq}{\hookrightarrow} \RusB_{\tilde{S}\cup \tilde{X}}(L, A) \stackrel{\simeq}{\twoheadleftarrow}\RusB_{\tilde{S}}(L, A) $$
    or
    $$(\F_2 \otimes A) \oplus \Sha^2_{\tilde{S} \cup \tilde{X}, 2}(L, A) \stackrel{\simeq}{\hookrightarrow} A \oplus \Sha^2_{\tilde{S} \cup \tilde{X}}(L, A) \stackrel{\simeq}{\hookrightarrow} A \oplus \RusB_{\tilde{S} \cup \tilde{X}}(L, A) \stackrel{\simeq}{\twoheadleftarrow} \RusB_{\tilde{S}}(L, A).$$
    Taking $\Gamma$-invariants and denoting by $S$ and $X$, respectively, the sets of primes of $K$ below $\tilde{S}$ and $\tilde{X}$, we obtain that either 
        $$\Sha^2(L_{\tilde{S}\cup \tilde{X}}(p)/K, A) \stackrel{\simeq}{\hookrightarrow} \Sha^2_{S \cup X}(K, A) \stackrel{\simeq}{\hookrightarrow} \RusB_{S\cup X}(K, A) \stackrel{\simeq}{\twoheadleftarrow} \RusB_S(K, A)$$
        or
        $$A^\Gamma \oplus \Sha^2(L_{\tilde{S} \cup \tilde{X}}(p)/K, A) \stackrel{\simeq}{\hookrightarrow} A^\Gamma \oplus \Sha^2_{S \cup X}(K, A) \stackrel{\simeq}{\hookrightarrow} 
        A^\Gamma \oplus \RusB_{S \cup X}(K, A) \stackrel{\simeq}{\twoheadleftarrow} \RusB_S(K, A),$$
    which proves part {\it (ii)}. 
    
    To prove the second part of  {\it (i)}, observe that $\dim A^\Gamma$ is equal to the multiplicity of the trivial representation in the decomposition of $A$.
    Thus, if $A$ has no {trivial} modules in its decomposition, then $\dim A^\Gamma = 0$, and the conclusion follows.    
\end{proof}

\section{Examples}\label{examples}

In \cite{MR4359923} several examples are given where the composite
$\Sha^2_{\emptyset, 2}(K,\F_2) \hookrightarrow \Sha^2_\emptyset(K, \F_2) \hookrightarrow \RusB_\emptyset(K,\F_2)$
is not surjective, but explicit sets  $S$ were given such that
\[
\Sha^2_{\emptyset, 2}(K,\F_2)\hookrightarrow \RusB_\emptyset(K,\F_2)
\simeq \RusB_S(K,\F_2) \stackrel{\simeq}{\hookleftarrow}\Sha^2_{S}(K,\F_2)
\stackrel{\simeq}{\hookleftarrow}\Sha^2_{S, 2}(K,\F_2).
\]
In this section we give two examples: one for odd $p$ and one for $p = 2$. We used MAGMA, \cite{MR1484478}, for our  computations. All code was run unconditionally, i.e. without assuming the GRH.
\begin{example}\label{example:One}
Let $K=\Q$, $L=\Q(\sqrt{5})$, $p=3$ and let $A$ be the group $\Z/3\Z$ with non-trivial action of
$\Gamma:=\Gal(L/K)$.
Note that $L/K$ is unramified outside $5$. 
Let $S=\{5\}$, $T=\{5,107\}$ and $V=\{5,107,197\}$ be sets of primes of $\Q$ and 
$\tilde{S}$, $\tilde{T}$ and $\tilde{V}$ are the corresponding sets of primes of $L$. 
Each rational prime of $V$ has a unique prime above it in $L$.

For $\tilde{X}\in \{ \tilde{S},\tilde{T},\tilde{V} \}$ we will study $L_{\tilde{X}}(3)$, the maximal pro-$3$ extension of $L$ unramified outside $\tilde{X}$. As $\tilde{X}$ is $\Gamma$-stable,  $L_{\tilde{X}}(3)/K$ is Galois.

A simple inflation-restriction argument, using that $(\#\Gamma,3)=1$, gives, for $X \in \{S,T,V\}$, that 
\[H^1(G_{K,X},A) \stackrel{\simeq}{\longrightarrow} H^1(G_{L,\tilde{X}},A)^\Gamma=\Hom_\Gamma(G_{L,\tilde{X}},A).
\]
The $3$-parts of the ray class groups are given below:
\begin{center}
{\it The $3$-parts of ray class groups for given tame conductors}
\end{center}

\begin{center}
\begin{tabular}{c|cccccc}
& $S$ & $T$ & $V$ & $\tilde{S}$ & $\tilde{T}$ & $\tilde{V}$\\ 
\hline
$K$ & $0$ & $0$ & $0$ &- & -&-\\
$L$ & -& -&-& $0$ &$\Z/3\Z$ & $\Z/27\Z$
\end{tabular}
\end{center}
The first row follows as none of the primes of $V$ are $1$ mod $3$, and the second  row comes from MAGMA computations. We want to determine the $\Gal(L/K)$-action on the ray class groups over $L$ for $\tilde{T}$ and $\tilde{V}$. That the first row is trivial implies the action of $\Gal(L/K)$ on the $3$-parts of the ray class groups over $L$ is multiplication by $-1$ so
\[
\Hom(G_{L,\tilde{X}},A)=\Hom_\Gamma(G_{L,\tilde{X}},A)=H^1(G_{L,\tilde{X}},A)^\Gamma=H^1(G_{K,X},A)
\]
and thus
\[
\dim H^1(G_{K,T},A) = \dim H^1(G_{L,\tilde{T}},A)=1
\mbox{ and }
\dim H^1(G_{K,V},A) = \dim H^1(G_{L,\tilde{V}},A)=1.
\]
As $G_{K,V} \twoheadrightarrow G_{K,T}$ (let $J$ denote the kernel) and the inflation map 
$H^1(G_{K,T},A) \longrightarrow H^1(G_{K,V},A)$ is an isomorphism, the latter half of the $5$-term inflation-restriction sequence gives the top row of:
\begin{center}
\begin{tikzcd} 
0  \arrow[r] & H^1(J,A)^{G_{K,T}}  \arrow[r] 
&H^2(G_{K,T},A) \arrow[r] \arrow[d]& H^2(G_{K,V},A)\arrow[d] \\
& & \oplus_{\q \in T} H^2(G_{K_\q},A) \arrow[r, hook] &   \oplus_{\q \in V} H^2(G_{K_{\q}},A)
\end{tikzcd}
\end{center}
From the second row of the table we have the $3$-part of the ray class group of conductor $\tilde{V}$ properly surjects onto that of conductor $\tilde{T}$, giving a non-trivial element 
$\alpha \in H^1(J,A)^{G_{K,T}}$. Chasing $\alpha$ to 
$\oplus_{{\q} \in V} H^2(G_{K_{\q}},A)$ its image is zero,  so its image 
in $\oplus_{\q \in T} H^2(G_{K_{\q}},A)$ is zero as well. Thus the image of $\alpha$ in  $H^2(G_{K,T},A)$ actually lies in $\Sha^2_T(K,A)$ so this latter group is non-zero.

It remains to show 
\begin{enumerate}[label=(\roman*)]
    \item  $\dim \RusB_S(K,A) =\dim \RusB_T(K,A) =1$,
    \item   $\Sha^2(L_{\tilde{S}}(3)/K,A)=0$.
\end{enumerate}
Once these are established, we will have shown that 
$\dim \Sha^2(L_{\tilde{X}}(3)/K,A)$ {\it increases} as we change $X$ from $S$ to $T$ and reaches its maximum of $\dim \RusB_S(K,A) =1$.

\begin{enumerate}[label=(\roman*)]
    \item By Lemma~\ref{lemma:equiv}{\it (iv)}, $ \RusB_S(K,A) \simeq \RusB_S(L,A)^\Gamma$ and  
$ \RusB_T(K,A) \simeq  \RusB_T(L,A)^\Gamma$.
Over $L$, the $\F_3[G_{L,\tilde{T}}]$-module $A$ has trivial action and is just 
$\F_3$. Using 
{the relation between $\dim H^1(G_{L, S}, \F_p)$ and $\dim \RusB_S(L, \F_p)$ given in (\ref{eq:ActualH1formula}) in the Introduction} and the table 
{above}, we have
$ \dim \RusB_S(L,A) =  \dim \RusB_T(L,A)=1$. Let $U_L$ be the unit group of $L$ and $\Cl_L[p]$ be the $p$-torsion in the class group of $L$.
{From ~(\ref{eq:Vemptyset}), we see
\[
0 \to (\Cl_L[3])^\vee \to \RusB_\emptyset(L,\F_3) \to (U_L/U_L^3)^\vee \to 0.
\]}
Using the fact that the class group of $L$ is trivial, its unit group has rank $1$, $\Gamma$ acts on $U_L/U^3_L$ by $-1$ and that $\Gamma$ acts on $A$ by $-1$ as well, {taking $\Gamma$-invariants gives} 
$ \dim \RusB_S({K},A) =  
\dim \RusB_T({K},A)=1$.
\item From the table we see there are no $3$-extensions of $L$ unramified outside $S$ so $L_{\tilde {S}}=L$. It follows that $\Sha^2(L_{\tilde{S}}(3)/K, A) \simeq \Sha^2_{\tilde{S}, 3}(L, A)^\Gamma = 0$. 
\end{enumerate}
\end{example}
\begin{example}\label{example:Two}

We adapt Example 1 of \cite{MR4359923}.
Let $K=\Q$ and $L=\Q(\zeta_7+\zeta^{-1}_7) = \Q(\cos \frac{2\pi}{7})$,
noting that $L$ is the splitting field of the polynomial
$x^3+x^2-2x-1$. This is the totally real subfield of $\Q(\zeta_7)$.
Clearly $\Gamma:=\Gal(L/\Q) \simeq \Z/3\Z$ and one can check
 that $L$ has trivial class group. 
We set $p=2$ and $A = \F_2\oplus \F_2$ with a non-trivial action of $\Gamma$. Note that 
$\{\pm 1\} \cap -4{\Q}^4=\emptyset$ so the situation is {\it not} exceptional, but our computations make no use of this.
\begin{lemma}\label{lemma:stable}
    $\dim_{\F_2} (A\otimes_{\F_2} A)^\Gamma=2$. 
\end{lemma}
\begin{proof}
    Extend scalars to $\F_4$ which contains all the eigenvalues of $\Gamma$ acting on $A$. 
  Let $\psi:\Gamma \to \F^{\times}_4$ be a non-trivial character.
  One checks
   $A\otimes_{\F_2} \F_4 \simeq \F_4(\psi) \oplus \F_4(\psi^{-1})$ so 
\[
 \dim_{\F_2} (A\otimes_{\F_2} A)^\Gamma=
\dim_{\F_4}\left[ \left(\F_4(\psi) \oplus \F_4(\psi^{-1})) \otimes_{\F_4} 
  (\F_4(\psi) \oplus \F_4(\psi^{-1})\right)\right]^\Gamma =2.
  \]
\end{proof}
Consider the following sets of primes of $K$  and $L$ below:
In  the sets of primes of $L$, the subscript ``$1$'' indicates the first prime of $L$ above the rational prime as chosen by MAGMA. 
$$ \mbox{Primes of $K$}: S=\{7\}, \mbox{ } T=\{7,181,293\}, \mbox{ and } V=\{7,181,293,307,349\}.$$
$$ \mbox{Primes of $L$}:S'=\{7_1\}, \mbox{ } T'=\{7_1,181_1,293_1\}, \mbox{ } V'=\{7_1,181_1,293_1,307_1,349_1\}.$$

The first row of the table below
gives the $2$-primary part of ray class groups of L with given tame conductor $X' \in \{ \emptyset, S', T', V' \}$, as computed by MAGMA. 
The second row follows from the first row and (\ref{eq:ActualH1formula}) of the Introduction.

\begin{center}
\begin{tabular}{c|cccc}
& $\emptyset$ & $S'$ & $T'$ & $V'$  \\ 
\hline
$RCG_{L,X'}[2^\infty]$ & $0$ & $0$ & $\Z/2\Z \times \Z/2\Z$ & $\Z/4\Z \times \Z/4\Z$\\
$\dim_{\F_2} \RusB_{X'}(L,\F_2)$ & $3$ & $2$ & $2$ & $0$\\
$\dim_{\F_2} \Sha^2_{X', 2}(L,\F_2)$ & $0$ & $0$ & $2$ & $0$\\
\end{tabular}
\end{center}
We now explain the third row: From the first row we see $L_\emptyset(2)=L_{S'}(2)=L$ so $G_{L,\emptyset}(2)=G_{L,S'}(2)$ is trivial which gives the first two entries of the third row. The fourth entry follows from the injection $\Sha^2_{V', 2}(L,\F_2) \hookrightarrow \RusB_{V'}(L,\F_2)=0$.
For the third entry,
using inflation-restriction for the map $1 \to N \to G_{L,V'}(2) \twoheadrightarrow G_{L,T'}(2) \to 1$, the last two entries of the first row of the table and the method of Example~\ref{example:One}, we have
$$H^1(N,\F_2)^{G_{L,T'}(2)} \hookrightarrow \Sha^2_{T',2}(L,\F_2).$$
As 
\begin{align*}
    2 &= \dim_{\F_2} \left( \ker \left(RCG_{L,V'}[2^\infty] \twoheadrightarrow RCG_{L,T'}[2^\infty]\right) \otimes_{\Z} \F_2 \right) \\
    & \leq \dim_{\F_2} H^1(N,\F_2)^{G_{L,T'}(2)} \leq \dim_{\F_2} \Sha^2_{T',2}(L,\F_2) \leq  \dim_{\F_2} \RusB_{T'} (L,\F_2)=2,
\end{align*}
we are done.

Let $\tilde{X}$ be the $\Gal(L/K)$-orbit of $X' \in \{S', T', V' \}$. We now justify the data in the table below:
\begin{center}
\begin{tabular}{c|cccc}
& $\emptyset$ & $\tilde{S}$ & $\tilde{T}$ & $\tilde{V}$  \\ 
\hline
$\dim_{\F_2} \RusB_{\tilde{X}}(L,A)^\Gamma$ & $2$ & $2$ & $2$ & $0$\\
$\dim_{\F_2} \Sha^2_{\tilde{X}, 2}(L,A)^\Gamma $ & $0$ & $0$ & $2$ & $0$\\
\end{tabular}
\end{center}
Since the class number of $L$ is $1$, we see 
$\RusB_\emptyset(L,\F_2) $ is the dual of $U_L/U_L^2$ where $U_L$ is the unit group of $L$, which has rank $2$ as $L$ is totally real. 
Recalling $\{\pm 1\} \subset U_L$,
one easily sees 
$U_L/U^2_L \simeq \F_2 \oplus A$ 
as $\F_2[\Gamma]$-modules   so
$\RusB_\emptyset(L,\F_2) \simeq \F_2 \oplus A$ as well.
Then 
$$\RusB_\emptyset(L,A) \simeq  \RusB_\emptyset(L,\F_2) \otimes_{\F_2} A \simeq A \oplus (A\otimes_{\F_2}A).$$
Taking $\Gamma$-invariants and using Lemma~\ref{lemma:stable} gives the first entry of the first row.
As $7$ is totally ramified in $L/K$, we see $S'=\tilde{S}$ so
$\dim_{\F_2} \RusB_{\tilde{S}}(L,\F_2)=2$. This group is a $\Gamma$-stable quotient of $\RusB_\emptyset(L,\F_2)\simeq\F_2 \oplus A$ and must therefore be $A$, so
$$\RusB_{\tilde{S}}(L,A) = \RusB_{\tilde{S}}(L,\F_2) \otimes_{\F_2} A \simeq A\otimes_{\F_2} A,$$
and Lemma~\ref{lemma:stable} gives the second entry of the first row. We know from the first table that
$\dim_{\F_2} \RusB_{T'}(L,\F_2)=2$ and  Lemma~\ref{RusBpreserved} gives
$\dim_{\F_2} \RusB_{\tilde{T}}(L,\F_2)=2$. As a $\Gamma$-module, this is again $A$, so tensoring with $A$ and taking $\Gamma$-invariants gives the third entry of the first row. The fourth entry  follows as 
$0=\RusB_{V'}(L,\F_2)=\RusB_{\tilde{V}}(L,\F_2)$ from the first table.

The first two entries of the second row follow from the fact that 
$L_\emptyset(2)=L_{\tilde{S}}(2)=L$ so $\Gal( L_\emptyset(2)/K)=
\Gal(L_{\tilde{S}}(2)/K)=\Gal(L/K)=\Gamma$. The fourth entry follows as 
$\Sha^2_{\tilde{V}, 2}(L,A) \hookrightarrow \RusB_{\tilde{V}}(L,A)=0$.
The third entry comes from the 
fact that the isomorphism 
$ \Sha^2_{T', 2}(L,\F_2) \stackrel{\simeq}{\hookrightarrow} \RusB_{T'}(L,\F_2)$ of $2$-dimensional vector spaces induces, using
Lemma~\ref{RusBpreserved},
the 
$\Gamma$-equivariant isomorphism 
$ \Sha^2_{\tilde{T}, 2}(L,\F_2) \stackrel{\simeq}{\hookrightarrow} \RusB_{\tilde{T}}(L,\F_2)$ of $2$-dimensional vector spaces. Both of these $\F_2[\Gamma]$-modules are $A$, so again, tensoring with $A$ and taking $\Gamma$-invariants completes this table.

From the second and third columns, we  have increased
$\Sha^2_{\tilde{X}, 2}(L,A)^\Gamma \simeq \Sha^2(L_{\tilde{X}}(2)/K,A)$ as we increase $X$ in size from $S$ to $T$ from the trivial group to 
$\RusB_{\tilde{T}}(L,A)^\Gamma \simeq \RusB_{\tilde{S}}(L,A)^\Gamma$, its maximal possible size.

\end{example}

\bibliographystyle{amsplain}
\bibliography{Shafarevich}

\end{document}